\documentclass[preprint,12pt]{elsarticle}

\usepackage{amsfonts,amssymb,amsmath,amssymb,mathtools,enumitem}
\usepackage{amsthm}
\usepackage{xcolor}
\usepackage[hidelinks]{hyperref}

\usepackage[capitalize]{cleveref}
\usepackage{overpic}

\newcommand{\rev}[1]{{#1}}

\newcommand{\R}{\mathbb{R}}
\newcommand{\C}{\mathbb{C}}
\newcommand{\K}{\mathbb{K}}

\newcommand{\F}{\textup{F}}
\DeclareMathOperator*{\argmin}{argmin}

\renewcommand{\d}{\,\textup{d}}

\newcommand*{\dd}{{\,\mathrm{d}}}
\newcommand*{\spec}{{\mathrm{Sp}}}
\renewcommand{\K}{\mathcal{K}}

\newcommand{\Av}{\mathbf{A}}
\newcommand{\Bv}{\mathbf{B}}
\newcommand{\Cv}{\mathbf{C}}

\newcommand{\Fv}{\mathbf{F}}
\newcommand{\Gv}{\mathbf{G}}

\newcommand{\Iv}{\mathbf{I}}

\newcommand{\Kv}{\mathbf{K}}

\newcommand{\Mv}{\mathbf{M}}

\newcommand{\Uv}{\mathbf{U}}
\newcommand{\Vv}{\mathbf{V}}
\newcommand{\Wv}{\mathbf{W}}
\newcommand{\Xv}{\mathbf{X}}
\newcommand{\Yv}{\mathbf{Y}}

\newcommand{\Psiv}{\mathbf{\Psi}}

\newcommand{\fv}{\mathbf{f}}
\newcommand{\gv}{\mathbf{g}}
\newcommand{\hv}{\mathbf{h}}

\newcommand{\vv}{\mathbf{v}}

\newcommand{\xv}{\mathbf{x}}
\newcommand{\yv}{\mathbf{y}}

\usepackage{bm}

\newtheorem{theorem}{Theorem}[section]
\newtheorem{lemma}[theorem]{Lemma}

\newtheorem{assumption}[theorem]{Assumption}
\crefname{assumption}{Assumption}{Assumptions}

\newtheorem{remark}[theorem]{Remark}
\numberwithin{equation}{section}

\journal{ArXiv}

\begin{document}

\begin{frontmatter}

    \title{On the Convergence of Hermitian Dynamic Mode Decomposition}

    \author[Imperial]{Nicolas Boull\'e}

    \affiliation[Imperial]{
        organization={Department of Mathematics, Imperial College London},
        city={London},
        postcode={SW7 2AZ},
        country={UK}}

    \author[Cambridge]{Matthew J. Colbrook}

    \affiliation[Cambridge]{
        organization={Department of Applied Mathematics and Theoretical Physics, University of Cambridge},
        city={Cambridge},
        postcode={CB3 0WA},
        country={UK}}

    \begin{abstract}
        We study the convergence of Hermitian Dynamic Mode Decomposition (DMD) to the spectral properties of self-adjoint Koopman operators. Hermitian DMD is a data-driven method that approximates the Koopman operator associated with an unknown nonlinear dynamical system, using discrete-time snapshots. This approach preserves the self-adjointness of the operator in its finite-dimensional approximations. \rev{We prove that, under suitably broad conditions, the spectral measures corresponding to the eigenvalues and eigenfunctions computed by Hermitian DMD converge to those of the underlying Koopman operator}. This result also applies to skew-Hermitian systems (after multiplication by $i$), applicable to generators of continuous-time measure-preserving systems. Along the way, we establish a general theorem on the convergence of spectral measures for finite sections of self-adjoint operators, including those that are unbounded, which is of independent interest to the wider spectral community. We numerically demonstrate our results by applying them to two-dimensional Schr\"odinger equations.
    \end{abstract}


\end{frontmatter}

\section{Introduction}

We consider discrete-time dynamical systems of the form:
\begin{equation} \label{eq:DynamicalSystem}
    \xv_{n+1} = \Fv(\xv_n), \quad n \in \mathbb{N},
\end{equation}
where $\xv\in\Omega$ denotes the state of the system, and $\Omega\subseteq\mathbb{R}^d$ is the state-space for $d\in\mathbb{N}$. The (typically) nonlinear function $\Fv:\Omega \rightarrow \Omega$ governs the system's evolution and is unknown. We assume that our knowledge of the system is limited to $M\geq 1$ discrete-time snapshots of the system, i.e., one has only access to a finite dataset of the form
\begin{equation} \label{eq_data}
    \left\{\xv^{(m)},\yv^{(m)}\right\}_{m=1}^M,\quad \text{such that}\quad \yv^{(m)}=\Fv(\xv^{(m)}),\quad 1\leq m\leq M.
\end{equation}
This snapshot data can be collected from either one long trajectory or multiple shorter trajectories and acquired via experimental observations or numerical simulations. In general, one aims to use this data to infer and reconstruct properties of the underlying dynamical system given by \eqref{eq:DynamicalSystem}. With the arrival of big data and machine learning, this data-driven viewpoint is currently undergoing a renaissance. Examples of applications of this framework arise naturally across many scientific fields, including fluid dynamics~\cite{schmid2010dynamic}, epidemiology~\cite{rowley2009spectral}, chemistry~\cite{narasingam2019koopman}, and neuroscience~\cite{brunton2016extracting}, to name a few.

One of the most prominent algorithms for data-driven analysis of dynamical systems is \textit{Dynamic Mode Decomposition} (DMD), which is closely connected with \textit{Koopman operators}. In 1931, Koopman introduced an operator-theoretic approach to dynamical systems, initially to describe Hamiltonian systems \cite{koopman1931hamiltonian}. This theory was further expanded by Koopman and von Neumann~\cite{koopman1932dynamical} to include systems with continuous spectra. A Koopman operator $\mathcal{K}$ lifts a nonlinear system \eqref{eq:DynamicalSystem} into an \textit{infinite-dimensional} space of observable functions $g:\Omega\rightarrow\mathbb{C}$ as
\[
    [\mathcal{K}g](\xv) = g(\Fv(\xv)), \quad \text{such that}\quad [\mathcal{K}g](\xv_n)=g(\xv_{n+1}) \text{ for } n\geq 0.
\]
Through this approach, the evolution dynamics become linear, enabling the use of generic solution techniques based on spectral decompositions. DMD was initially developed in the context of fluid dynamics~\cite{schmid2008dynamic,schmid2010dynamic}. Earlier, Mezi{\'c} introduced the Koopman mode decomposition~\cite{mezic2005spectral}, providing a theoretical basis for Rowley et al. to connect DMD with Koopman operators~\cite{rowley2009spectral}. However, the standard DMD algorithm is based on linear observables and generally fails to capture truly nonlinear phenomena. To address this limitation, \textit{Extended DMD} (EDMD)~\cite{williams2015data} extends the DMD algorithm to nonlinear observables. Then, under suitable conditions, in the large-data limit $M\rightarrow\infty$, EDMD converges to the numerical approximation obtained by a Galerkin method in the limit of large data sets. At its core, EDMD is a projection method that aims to compute the spectral properties of Koopman operators. \rev{It is important to realize that EDMD only converges in the strong operator topology to $\mathcal{K}$ \cite{korda2018convergence}\footnote{Essentially, this means that the action of EDMD on observables converges to the action of the Koopman operator on observables. This form of convergence implies that limit points of eigenpairs of EDMD are also eigenpairs of the Koopman operator if the limiting vector is non-zero \cite[Theorem 4]{korda2018convergence}. As the authors of \cite{korda2018convergence} highlight, this is an extremely weak form of convergence.} and its spectral properties need not converge \cite[Example 2]{mezic2020numerical}. For ways of overcoming this and classifications of how difficult Koopman spectral computations are, see \cite{colbrook2024limits}.}

In a recent paper \cite{baddoo2023physics}, Baddoo et al. provided a unified framework, \textit{Physics-Informed Dynamic Mode Decomposition} (piDMD), for imposing physical constraints in DMD with a linear choice of dictionary. For some constraints, this framework can be extended to the setting of EDMD, thus providing an approximation of Koopman operators in the spirit of geometric integration \cite{hairer2010geometric}. A scenario where it is possible to extend piDMD to nonlinear observables occurs when the Koopman operator associated with the dynamical system is Hermitian, and one aims to compute finite-dimensional Hermitian approximations to preserve spectral properties. The corresponding algorithm, known as \textit{Hermitian Dynamic Mode Decomposition}, has been studied extensively by Drma{\v{c}}~\cite{drmavc2024hermitian}. However, there is currently no proof of convergence of this method to the spectral properties of the underlying Koopman operator, despite encouraging numerical results~\cite{baddoo2023physics,drmavc2024hermitian}.

This note addresses this gap by showing the convergence of \textit{Hermitian Dynamic Mode Decomposition} to the spectral properties of self-adjoint Koopman operators. Along the way, we derive \cref{thm_gen_convergence}, a result that may be of independent interest in the wider spectral community. For methods that compute spectral measures of general self-adjoint systems, see \cite{colbrook2019computing,colbrook2020computing}. All of our results naturally extend to the case when the Koopman operator is skew-Hermitian, another structure of broad interest. As an example, Koopman generators of invertible measure-preserving continuous-time dynamical systems are skew-adjoint. Hence, our analysis carries over to the application of (skew) Hermitian DMD to DMD methods that approximate such generators \cite{klus2020data}. Our results can also be extended to stochastic dynamical systems and the stochastic Koopman operator \cite{mezic2000comparison,vcrnjaric2020koopman,wanner2022robust,colbrook2023beyond}.

It is worth pointing out the myriad of papers on Koopman operators and DMD, as evidenced by various surveys~\cite{budivsic2012applied,brunton2021modern,colbrook2023multiverse,mezic2013analysis,otto2021koopman,schmid2022dynamic}. For example, the survey \cite{colbrook2023multiverse} characterizes structure-preserving methods as one of the flavors of DMD. Despite this widespread interest, convergence results pertaining to the relevant spectral properties of Koopman operators remain decidedly rare. Exceptions to this rule include methods with theoretical guarantees, such as \textit{Hankel-DMD}~\cite{arbabi2017ergodic}, \textit{Residual DMD}~\cite{colbrook2023residualJFM,colbrook2021rigorous}, \textit{Rigged DMD} \cite{colbrook2024rigged}, \textit{Measure-Preserving EDMD}~\cite{colbrook2023mpedmd}, compactification methods~\cite{das2021reproducing,valva2023consistent}, and periodic approximation methods~\cite{govindarajan2019approximation,govindarajan2021approximation}. While related to the present note, the latter four approaches assume that the system is measure-preserving, which differs from the setup addressed here. Given the keen interest in convergence results, we hope this note will encourage further exploration into the conditions under which methods like piDMD converge.

The rest of the paper is organized as follows. We provide preliminaries in \cref{sec_preliminary} on Koopman operators and EDMD. Then, in \cref{sec_self_adjoint}, we recall and derive Hermitian DMD in the context of EDMD. Finally, \cref{sec:convergence_results} contains our main convergence results, which are demonstrated numerically in \cref{sec:computational_example}.

\section{Preliminaries} \label{sec_preliminary}

In this section, we provide preliminaries on Koopman operators, EDMD, and the role of spectral measures for self-adjoint Koopman operators.

\subsection{Koopman operators}
\label{sec:Koopman_operators}

To define a Koopman operator, we begin with a space $\mathcal{F}$ of functions $g: \Omega \rightarrow \mathbb{C}$, where $\Omega$ is the state space of our dynamical system. The functions $g$, referred to as \textit{observables}, serve as tools for indirectly measuring the state of the system described in \eqref{eq:DynamicalSystem}. Specifically, $g(\xv_n)$ indirectly measures the state $\xv_n$. Koopman operators enable us to capture the time evolution of these observables through a linear operator framework. For a suitable domain $\mathcal{D}(\mathcal{K}) \subset \mathcal{F}$, we define the Koopman operator via the composition formula:
\begin{equation} \label{eq:KoopmanOperator}
    [\mathcal{K}g](\xv) = [g\circ \Fv](\xv)=g(\Fv(\xv)), \quad g\in \mathcal{D}(\mathcal{K}).
\end{equation}
In this context, $[\mathcal{K}g](\xv_n)= g(\Fv(\xv_n))=g(\xv_{n+1})$ represents the measurement of the state one time step ahead of $g(\mathbf{x}_n)$. This process effectively captures the dynamic progression of the system. The overarching concept is summarized in \cref{schematic}.

\begin{figure}[htbp]
    \centering
    \includegraphics[width=0.9\textwidth,trim={0mm 0mm 0mm 0mm},clip]{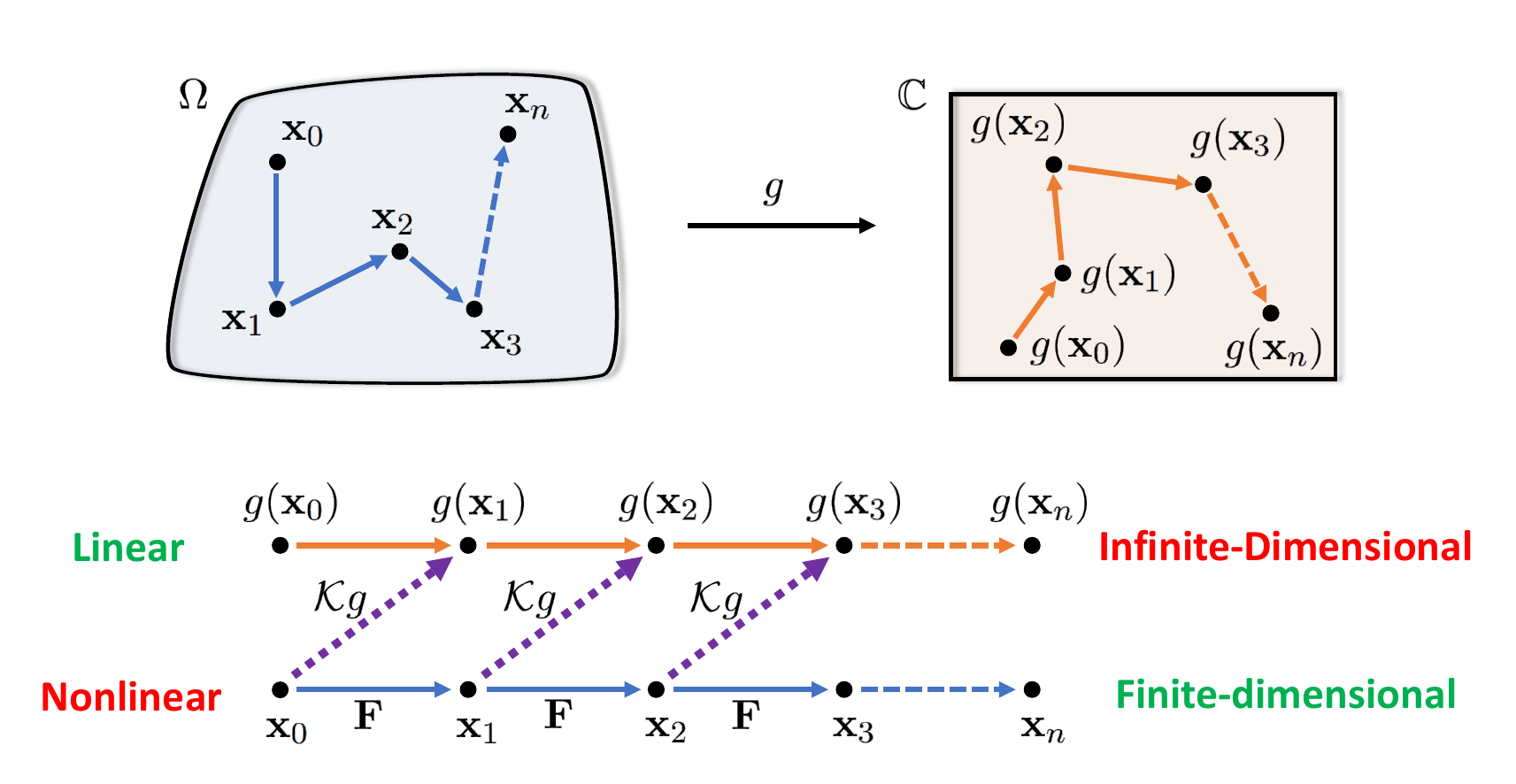}
    \caption{Summary of the idea of Koopman operators. By lifting to a space of observables, we trade a nonlinear finite-dimensional system for a linear infinite-dimensional system.}
    \label{schematic}
\end{figure}

The key property of the Koopman operator $\mathcal{K}$ is its \textit{linearity}. This linearity holds irrespective of whether the system's dynamics, as represented in \eqref{eq:DynamicalSystem}, are linear or nonlinear. Consequently, the spectral properties of $\mathcal{K}$ become a powerful tool to analyze the dynamical system's behavior. We focus on cases where $\mathcal{F}=L^2(\Omega,\omega)$ is a Hilbert space with the following inner product
$$
    \langle g_1,g_2 \rangle=\int_{\Omega} g_1(\xv)\overline{g_2(\xv)}\ \mathrm{d}\omega(\xv),
$$
for some positive measure $\omega$. In going from a pointwise definition in \eqref{eq:KoopmanOperator} to the space $L^2(\Omega,\omega)$, a little care is needed since $L^2(\Omega,\omega)$ consists of equivalence classes of functions. We assume that the map $\Fv$ is nonsingular with respect to $\omega$, meaning that
$$
    \omega(E)=0\quad \text{implies}\quad \omega(\{\xv:\Fv(\xv)\in E\})=0.
$$
This ensures that the Koopman operator is well-defined since $g_1(\xv)=g_2(\xv)$ for $\omega$-almost every $\xv$ implies that $g_1(\Fv(\xv))=g_2(\Fv(\xv))$ for $\omega$-almost every $\xv$.

The above Hilbert space setting is standard in most of the Koopman literature. It is crucial to recognize that the Koopman operator is not uniquely defined by the dynamical system in \eqref{eq:DynamicalSystem}; rather, it is fundamentally dependent on the choice of the space of observables $\mathcal{F}$. Since $\mathcal{K}$ acts on an \textit{infinite-dimensional} function space, we have exchanged the nonlinearity in \eqref{eq:DynamicalSystem} for an infinite-dimensional linear system. This means that the spectral properties of $\mathcal{K}$ can be significantly more complex than those of a finite matrix, making them more challenging to compute \cite{colbrook2020PhD,ben2015can}.

\subsection{Extended dynamic mode decomposition} \label{sec:EDMD}

The objective of EDMD is to approximate the action of the Koopman operator $\mathcal{K}$ on a finite-dimensional vector space of functions by a matrix $\Kv$. For the sake of simplicity, the initial formulation of EDMD assumes that the sample points $\{\mathbf{x}^{(m)}\}_{m=1}^M$ in the snapshot dataset are independently sampled from the distribution $\omega$. In this section, we consider the points $\xv^{(m)}$ as \textit{quadrature nodes} used for integration with respect to the measure $\omega$. This adaptability enables the choice of quadrature weights tailored to different scenarios.

One first chooses a dictionary of functions $\{\psi_1,\ldots,\psi_{N}\}$, i.e., a list of observables, in the space $\mathcal{D}(\mathcal{K})\subset L^2(\Omega,\omega)$. These observables form a finite-dimensional subspace $V_N=\mathrm{span}\{\psi_1,\ldots,\psi_{N}\}$. EDMD consists of selecting a matrix $\Kv\in\mathbb{C}^{N\times N}$ that approximates the action of $\mathcal{K}$ confined to this subspace, such that for $1\leq j\leq N$ we have,
\[
    [\mathcal{K}\psi_j](\xv) = \psi_j(\Fv(\xv)) \approx \sum_{i=1}^{N} \Kv_{ij} \psi_i(\xv).
\]
We define the vector-valued feature map, $\Psiv:\Omega\to \C^N$, as
\begin{equation}
    \label{feat_map_def}
    \Psiv(\xv)=\begin{bmatrix}\psi_1(\xv) & \cdots& \psi_{{N}}(\xv) \end{bmatrix}\in\mathbb{C}^{1\times {N}},\quad \xv\in\Omega.
\end{equation}
Any function $g\in V_{N}$ can be expressed as a linear combination of the basis functions as $g(\xv)=\sum_{j=1}^{N}\psi_j(\xv)\gv_j=\Psiv(\xv)\,\gv$, for some vector $\gv\in\mathbb{C}^{N}$. Therefore,
\[
    [\mathcal{K}g](\xv)=\Psiv(\Fv(\xv))\,\gv=\Psiv(\xv)(\Kv\,\gv)+\underbrace{\sum_{j=1}^{N}\psi_j(\Fv(\xv))\gv_j-\Psiv(\xv)(\Kv\,\gv)}_{R(\gv,\xv)}.
\]
In general, $V_{N}$ is not an invariant subspace of $\mathcal{K}$. Hence, there is no choice of $\Kv$ that makes the residual $R(\gv,\xv)$ zero for all $g\in V_N$ and $\omega$-almost every $\xv\in\Omega$. Instead, it is natural to select the matrix $\Kv$ to minimize the residual:
\begin{equation}\label{eq:ContinuousLeastSquaresProblem}
    \begin{aligned}
        \Kv & = \argmin_{\Kv\in\mathbb{C}^{N\times N}} \int_\Omega \max_{\substack{\gv\in\mathbb{C}^{N}                                                             \\\|\Cv\gv\|_{\ell^2}=1}}|R(\gv,\xv)|^2\ \mathrm{d}\omega(\xv)\\
            & =\argmin_{\Kv\in\mathbb{C}^{N\times N}} \int_\Omega \left\|\Psiv(\Fv(\xv))\Cv^{-1} - \Psiv(\xv)\Kv\Cv^{-1}\right\|^2_{\ell^2}\ \mathrm{d}\omega(\xv).
    \end{aligned}
\end{equation}
Here, $\Cv\in\R^{N\times N}$ is a positive self-adjoint matrix that controls the size of $g=\Psiv\gv$. One should interpret the matrix $\Cv$ as choosing an appropriate norm. This becomes important when $N\rightarrow\infty$ since not all the norms defined on an infinite-dimensional vector space are equivalent.

While it is not possible to directly evaluate the integral in \eqref{eq:ContinuousLeastSquaresProblem} from the snapshot data, one can instead approximate it via a quadrature rule with nodes $\{\xv^{(m)}\}_{m=1}^{M}$ and weights $\{w_m\}_{m=1}^{M}$. For notational convenience, we introduce the weight matrix $\Wv=\mathrm{diag}(w_1,\ldots,w_{M})$ and the matrices
\[
    \Psiv_X=\begin{pmatrix}
        \Psiv(\xv^{(1)}) \\
        \vdots           \\
        \Psiv(\xv^{(M)})
    \end{pmatrix}\in\mathbb{C}^{M\times N},\quad\text{and}\quad
    \Psiv_Y=\begin{pmatrix}
        \Psiv(\yv^{(1)}) \\
        \vdots           \\
        \Psiv(\yv^{(M)})
    \end{pmatrix}\in\mathbb{C}^{M\times N}.
\]
After discretizing \eqref{eq:ContinuousLeastSquaresProblem}, one obtains the following weighted least-squares problem:
\begin{equation} \label{EDMD_opt_prob2}
    \begin{aligned}
        \Kv & = \argmin_{\Kv\in\mathbb{C}^{N\times N}}\sum_{m=1}^{M} w_m\left\|\Psiv(\yv^{(m)})\Cv^{-1}-\Psiv(\xv^{(m)})\Kv\Cv^{-1}\right\|^2_{2} \\
            & =\argmin_{\Kv\in\mathbb{C}^{N\times N}} \left\|\Wv^{1/2}\Psiv_Y\Cv^{-1}-\Wv^{1/2}\Psiv_X \Kv\Cv^{-1}\right\|_{\mathrm{F}}^2,
    \end{aligned}
\end{equation}
where $\|\cdot\|_{\mathrm{F}}$ denotes the Frobenius norm. By reducing the size of the dictionary if necessary, we may assume without loss of generality that $\Wv^{1/2}\Psiv_X$ has rank $N$. A solution to \eqref{EDMD_opt_prob2} is given by
\[
    \Kv= (\Wv^{1/2}\Psiv_X)^\dagger \Wv^{1/2}\Psiv_Y=(\Psiv_X^*\Wv\Psiv_X)^\dagger\Psiv_X^*\Wv\Psiv_Y,
\]
where $\dagger$ denotes the Moore--Penrose pseudoinverse. The second equality follows since $\Wv^{1/2}\Psiv_X$ has linearly independent columns. Note that this solution is independent of the matrix $\Cv$. However, a suitable choice of $\Cv$ is vital once constraints on the matrix $\Kv$ are added to the optimization problem in \eqref{eq:ContinuousLeastSquaresProblem}~\cite{colbrook2023mpedmd}. In the case where the quadrature weights are equal and $\Psiv$ is the state (i.e., a linear dictionary), then $\Kv^\top$ is the transpose of the classical DMD matrix.

We now define the two correlation matrices $\Gv\in \C^{N\times N}$ and $\Av\in \C^{N\times N}$:
\begin{equation}
    \label{eq_EDMD_corr_matrices}
    \begin{aligned}
        \Gv & =\Psiv_X^*\Wv\Psiv_X=\sum_{m=1}^{M} w_m \Psiv(\xv^{(m)})^*\Psiv(\xv^{(m)}), \\
        \Av & =\Psiv_X^*\Wv\Psiv_Y=\sum_{m=1}^{M} w_m \Psiv(\xv^{(m)})^*\Psiv(\yv^{(m)}).
    \end{aligned}
\end{equation}
If we consider the discrete measure $\omega_M=\sum_{m=1}^Mw_m\delta_{\xv^{(m)}}$, then
\[
    \Gv_{jk}=\int_{\Omega} \overline{\psi_j(\xv)}\psi_k(\xv) \ \mathrm{d} \omega_M(\xv),\quad \Av_{jk}=\int_{\Omega} \overline{\psi_j(\xv)}\psi_k(\Fv(\xv)) \ \mathrm{d} \omega_M(\xv).
\]
Additionally, if the quadrature discretization converges as the number of data points $M\to\infty$, then we have
\begin{equation}
    \label{quad_convergence}
    \lim_{M\rightarrow\infty}\Gv_{jk} = \langle \psi_k,\psi_j \rangle,\quad \text{and}\quad \lim_{M\rightarrow\infty}\Av_{jk} = \langle \mathcal{K}\psi_k,\psi_j \rangle.
\end{equation}
Therefore, in the large data limit, $\Kv=\Gv^\dagger\Av$ approaches a matrix representation of the operator $\mathcal{P}_{V_{N}}\mathcal{K}\mathcal{P}_{V_{N}}^*$, where $\mathcal{P}_{V_{N}}:L^2(\Omega,\omega)\rightarrow V_{N}$ denotes the orthogonal projection onto the subspace $V_{N}$. In essence, EDMD is a Galerkin method. The EDMD eigenvalues thus approach the spectrum of $\mathcal{P}_{V_{N}}\mathcal{K}\mathcal{P}_{V_{N}}^*$, and EDMD is an example of the so-called \textit{finite section method} \cite{bottcher1983finite}.

\subsection{Spectral measures of self-adjoint Koopman operators}

If $g\in L^2(\Omega,\omega)$ is an \textit{eigenfunction} of $\mathcal{K}$ with \textit{eigenvalue} $\lambda$, then $g$ exhibits perfect coherence\footnote{In the setting of dynamical systems, coherent sets or structures are subsets of the phase space where elements (e.g., particles, agents, etc.) exhibit similar behavior over some time interval. This behavior remains relatively consistent despite potential perturbations or the chaotic nature of the system.} with
\begin{equation}
    \label{eq:perfectly_coherent}
    g(\xv_n)=[\mathcal{K}^ng](\xv_0)=\lambda^n g(\xv_0),\quad n\in\mathbb{N}.
\end{equation}
The oscillation and decay/growth of the observable $g$ are dictated by the complex argument and absolute value of the eigenvalue $\lambda$, respectively. In infinite dimensions, the appropriate generalization of the set of eigenvalues of $\mathcal{K}$ is the \textit{spectrum}, denoted by $\mathrm{Sp}(\mathcal{K})$, and defined as
$$
    \mathrm{Sp}(\mathcal{K})=\left\{z\in\mathbb{C} :(\mathcal{K} - zI)^{-1}\text{ does not exist as a bounded operator}\right\}\subset\mathbb{C}.
$$
Here, $I$ denotes the identity operator. The spectrum $\mathrm{Sp}(\mathcal{K})$ includes the set of eigenvalues of $\mathcal{K}$, but in general, $\mathrm{Sp}(\mathcal{K})$ contains points that are not eigenvalues. This is because there are more ways for $(\mathcal{K} - zI)^{-1}$ to not exist in infinite dimensions than in finite dimensions. For example, we may have continuous spectra.

From now on, we assume that $\mathcal{K}$ is a self-adjoint operator acting on $L^2(\Omega,\omega)$. Under this condition, the spectral theorem\footnote{For readers unfamiliar with the spectral theorem, \cite{halmos1963does} provides an excellent and readable introduction.}~\cite[Thm.~X.4.11]{conway2019course} allows us to diagonalize the Koopman operator $\mathcal{K}$, and its spectrum $\mathrm{Sp}(\mathcal{K})$ lies within the real line $\mathbb{R}$ since $\mathcal{K}$ is self-adjoint. There is a \textit{projection-valued measure} $\mathcal{E}$ supported on $\mathrm{Sp}(\mathcal{K})$, which associates an orthogonal projector with each Borel measurable subset of $\mathbb{R}$. For such a subset $S\subset\mathbb{R}$, $\mathcal{E}(S)$ is a projection onto the spectral elements of $\mathcal{K}$ inside $S$. For any observable $g\in \mathcal{D}(\mathcal{K})$,
\[
    g=\left(\int_\mathbb{R} \ \mathrm{d}\mathcal{E}(\lambda)\right)g \quad\text{and}\quad \mathcal{K}g=\left(\int_\mathbb{R} \lambda\ \mathrm{d}\mathcal{E}(\lambda)\right)g.
\]
The essence of this formula is the decomposition of $g$ according to the spectral content of $\mathcal{K}$. The projection-valued measure $\mathcal{E}$ simultaneously decomposes the space $L^2(\Omega,\omega)$ and diagonalizes the Koopman operator. For example, if $g\in\mathcal{D}(\mathcal{K}^n)$, we have
\begin{equation}
    \label{KMD_spec_meas_forward}
    g(\xv_n)= [\mathcal{K}^ng](\xv_0)= \left[\left(\int_\mathbb{R} \lambda^n\ \mathrm{d}\mathcal{E}(\lambda)\right)g\right](\xv_0).
\end{equation}
The spectral theorem offers a custom Fourier-type transform specifically for the operator $\mathcal{K}$ that extracts coherent features. \textit{Scalar-valued} spectral measures are of particular interest. Hence, given a normalized observable $g\in L^2(\Omega,\omega)$ with $\|g\| = 1$, the scalar-valued spectral measure of $\mathcal{K}$ with respect to $g$ is a probability measure defined as
\[
    \mu_g(S)=\langle \mathcal{E}(S)g,g \rangle.
\]
The spectral measure of $\mathcal{K}$ with respect to $g\in L^2(\Omega,\omega)$ is a signature for the forward-time dynamics of \eqref{eq:DynamicalSystem}. \rev{In this paper, we show that the spectral measures computed by Hermitian DMD converge to those of $\mathcal{K}$.}

\section{Hermitian Dynamic Mode Decomposition} \label{sec_self_adjoint}

When the Koopman operator $\K$ is Hermitian, i.e., $\mathcal{K}=\mathcal{K}^*$, a natural constraint is to preserve the Hermitian property on its finite-dimensional approximation, $\Kv$. This ensures that the spectral properties of $\Kv$ are consistent with those of $\K$ as the size of the dictionary increases. Generally, the solution to \eqref{EDMD_opt_prob2} is not Hermitian, and diverse strategies have been proposed to enforce this constraint~\cite{drmavc2024hermitian}. Here, we consider the Hermitian DMD algorithm introduced in~\cite{baddoo2023physics}.

First, given the Gram matrix $\Gv=\Psiv_X^*\Wv\Psiv_X$, one can approximate the inner product $\langle \cdot,\cdot\rangle$ on $L^2(\Omega,\omega)$ via the inner product induced by $\Gv$ as
\begin{equation} \label{inner_product_G}
    \langle \Psiv \gv,\Psiv \hv \rangle=\sum_{j,k=1}^N  \overline{h_j}g_k\langle \psi_k,\psi_j \rangle\approx\sum_{j,k=1}^N  \overline{h_j}g_k\Gv_{j,k}= \hv^*\Gv\gv.
\end{equation}
If \eqref{quad_convergence} holds, then this approximation converges to the inner product on $L^2(\Omega,\omega)$ in the large data limit as $M\rightarrow\infty$. We follow the EDMD approach described in \cref{sec:EDMD} but enforce the additional constraint that the matrix representation of the Koopman operator is self-adjoint with respect to the inner product induced by the matrix $\Gv$. Hence, we consider the following constrained least-square problem:
\begin{align*}
    \Kv & = \argmin_{\substack{\Kv\in \C^{N\times N} \\\Gv\Kv=\Kv^*\Gv}}\sum_{m=1}^M \left\|\Psiv(\yv^{(m)})\Gv^{-1/2}-\Psiv(\xv^{(m)})\Kv \Gv^{-1/2}\right\|_2^2\\
        & = \argmin_{\substack{\Kv\in \C^{N\times N} \\\Gv\Kv=\Kv^*\Gv}}\left\|\Wv^{1/2}\Psiv_Y\Gv^{-1/2}-\Wv^{1/2}\Psiv_X\Kv \Gv^{-1/2}\right\|_\F^2.
\end{align*}
After performing the change of variables $\Bv=\Gv^{1/2}\Kv \Gv^{-1/2}$, we obtain
\begin{equation}
    \label{symmetric_easy}
    \min_{\substack{\Bv\in \C^{N\times N}\\\Bv^*=\Bv}}\left\|\Wv^{1/2}\Psiv_Y \Gv^{-1/2}-\Wv^{1/2}\Psiv_X \Gv^{-1/2}\Bv\right\|_\F^2.
\end{equation}
Here, we recognize a symmetric Procrustes problem~\cite{baddoo2023physics,higham1988symmetric} of the form
$$
    \min_{\substack{\Mv\in \C^{N\times N}\\\Mv^*=\Mv}}\left\|\Yv-\Xv\Mv\right\|_\F^2.
$$
A solution can be computed from the economized singular value decomposition of the matrix $\Xv=\Uv\pmb{\Sigma} \Vv^*$, where $\pmb{\Sigma}$ is the $N\times N$ diagonal matrix containing the singular values $\sigma_1\geq \sigma_2\geq \cdots\geq \sigma_N\geq 0$. The solution is given by the matrix
$\Mv = \Vv \pmb{\Upsilon} \Vv^*$, where the entries of Hermitian matrix $\pmb{\Upsilon}\in \C^{N\times N}$ are defined by
$$
    \pmb{\Upsilon}_{ij} = \begin{cases}
        \frac{\sigma_i c_{ij}+\sigma_j \overline{c_{ji}}}{\sigma_i^2+\sigma_j^2}, & \text{if } \sigma_i^2 +\sigma_j^2 \neq 0, \\
        0,                                                                        & \text{otherwise},
    \end{cases}
$$
for $1\leq i,j,\leq N$. Here, the coefficients $c_{ij}$ are the entries of the matrix $\Cv=\Uv^* \Yv \Vv$. As observed by Higham~\cite{higham1988symmetric}, this algorithm is backward stable following the backward stability of the Golub--Reinsch SVD algorithm~\cite[Sec.~5.5.8]{golub1996matrix}.

We make an additional observation compared to the original formulation in~\cite{baddoo2023physics}, which considerably simplifies the algorithm to solve \eqref{symmetric_easy}. In the case of \eqref{symmetric_easy}, we have $\Xv=\Wv^{1/2}\Psiv_X \Gv^{-1/2}$ and $\Yv=\Wv^{1/2}\Psiv_Y \Gv^{-1/2}$, and the symmetric Procrustes algorithm can be simplified since $\Xv^*\Xv=\Iv$, the identity matrix. Hence, one can select $\pmb{\Sigma}$ and $\Vv$ to be the identity matrices, and the solution becomes
\begin{equation} \label{eq_new_procrustes}
    \Kv=\Gv^{-1/2}\Bv\Gv^{1/2}=\Gv^{-1}\left(\frac{\Av+\Av^*}{2} \right).
\end{equation}
Therefore, if the convergence in \eqref{quad_convergence} holds, the large data limit of Hermitian DMD becomes a Galerkin approximation of the Koopman operator since $\langle \mathcal{K}\psi_k,\psi_j \rangle=\langle \psi_k,\mathcal{K}\psi_j \rangle$. Additionally, the formula given by \eqref{eq_new_procrustes} is more computationally efficient than the standard symmetric Procrustes approach employed in~\cite{baddoo2023physics} as it avoids the computation of the SVD of $\Xv$, as well as a matrix square-root, which can be numerically unstable when $\Gv$ is ill-conditioned.

Given this convergence result, it is natural to study the convergence of this approximation as the size of the dictionary increases. It is well-known that spectra of Galerkin approximations of operators can suffer from discretization issues such as spectral pollution (spurious eigenvalues), spectral invisibility (missing parts of the spectrum), instabilities, and so forth \cite{colbrook2019compute,colbrook4,colbrook3}. These issues occur and, in fact, can worsen as the dictionary's size increases. Hence, the convergence of the Hermitian DMD algorithm might be difficult to establish. Nevertheless, we will show in \cref{sec:convergence_results} that the (scalar-valued) spectral measures converge weakly, thus providing theoretical guarantees for Hermitian DMD. One crucial property exploited in the proof is the self-adjointness of the finite-dimensional approximations, as imposed by Hermitian DMD.

\section{A general convergence result}
\label{sec:convergence_results}

Throughout this section, we consider an arbitrary self-adjoint operator $\mathcal{L}$ acting on a Hilbert space $\mathcal{H}$ with domain $\mathcal{D}(\mathcal{L})$. \rev{For example, the differential operator $-\mathrm{d}^2/\mathrm{d}x^2$ on $L^2(\mathbb{R})$ is self-adjoint with domain $\{g\in L^2(\mathbb{R}):g''\in L^2(\mathbb{R})\}$. This is an example of an unbounded operator (recall that an operator $A$ is bounded if $\sup_{x\in\mathcal{D}(A),\|x\|=1}\|Ax\|<\infty$). Moreover, the operators in this section need not be Koopman operators. } Let $\{\mathcal{P}_n:\mathcal{H}\rightarrow \mathcal{V}_n\}_{n\in\mathbb{N}}$ be a sequence of orthogonal projections onto a Hilbert space $\mathcal{V}_n\subset \mathcal{H}$, such that $\mathcal{P}_n^*\mathcal{P}_n$ converges strongly to the identity, meaning that
$$
    \lim_{n\rightarrow \infty}\|\mathcal{P}_n^*\mathcal{P}_n u-u\|=0,\quad u\in\mathcal{H}.
$$
We also assume that $\mathcal{V}_n\subset\mathcal{D}(\mathcal{L})$. We are interested in the spectral measure of $\mathcal{L}$ with respect to a vector $v\in \mathcal{H}$, which we denote by $\mu_v$. We let $\mu_{v,n}$ represent the scalar-valued spectral measure of $\mathcal{P}_n\mathcal{L}\mathcal{P}_n^*$ with respect to the vector $\mathcal{P}_nv\in\mathcal{V}_n$. We will show that the spectral measures $\mu_{v,n}$ converge weakly to $\mu_v$. While this is a well-known result for bounded self-adjoint operators, the standard proof does not carry over to the unbounded case, which requires a different approach.

\subsection{Bounded operators - the easy case}

The case when $\mathcal{L}$ is bounded is well-known. The following standard proof of convergence essentially boils down to a moment-matching procedure.

\begin{lemma}
    \label{lem:FS_moments}
    Consider the above setup and suppose that $\mathcal{L}$ is bounded. Then, for any bounded, continuous, function $\phi$ on $\mathbb{R}$ and for any $v\in\mathcal{H}$,
    \begin{equation}\label{eq:meaning_weak_strong}
        \lim_{n\rightarrow\infty}\int_\mathbb{R} \phi(\lambda)\dd\mu_{v,n}(\lambda)=\int_\mathbb{R} \phi(\lambda)\dd\mu_v(\lambda).
    \end{equation}
    In particular, $\mu_{v,n}$ converges weakly to $\mu_v$ as $n\rightarrow\infty$.
\end{lemma}

\begin{proof}
    Since $\|\mathcal{P}_n\mathcal{L}\mathcal{P}_n^*\|\leq \|\mathcal{L}\|$, $\spec(\mathcal{P}_n\mathcal{L}\mathcal{P}_n^*)$ is contained in some finite fixed interval. Since the support of $\mu_{v,n}$ is contained in $\spec(\mathcal{P}_n\mathcal{L}\mathcal{P}_n^*)$, without loss of generality, we assume that $ \phi $ is supported on a finite interval. Since any such function can be approximated to arbitrary accuracy by a polynomial, it is enough to prove \eqref{eq:meaning_weak_strong} for $\phi(\lambda)=\lambda^k$, for any $k\in\mathbb{Z}_{\geq 0}$. In other words, it is enough to show that the moments of the measures converge.

    The functional calculus shows that
    $$
        \int_\mathbb{R} \lambda^k\dd\mu_{v,n}(\lambda)=\langle\mathcal{P}_n^*(\mathcal{P}_n\mathcal{L}\mathcal{P}_n^*)^k \mathcal{P}_n v,v\rangle,\quad \int_\mathbb{R} \lambda^k\dd\mu_v(\lambda)=\langle \mathcal{L}^k v,v\rangle.
    $$
    Since $\mathcal{L}$ is bounded and $\mathcal{P}_n^*\mathcal{P}_n$ converges strongly to the identity,
    $$
        \mathcal{P}_n^*(\mathcal{P}_n\mathcal{L}\mathcal{P}_n^*)^k \mathcal{P}_n=\mathcal{P}_n^*\mathcal{P}_n(\mathcal{L}\mathcal{P}_n^*\mathcal{P}_n)^k
    $$
    converges strongly to $\mathcal{L}^k$ for any $k\in\mathbb{Z}_{\geq 0}$. It follows that $\mathcal{P}_n^*(\mathcal{P}_n\mathcal{L}\mathcal{P}_n^*)^k \mathcal{P}_n v$ converges weakly to $\mathcal{L}^k v$. The convergence in \eqref{eq:meaning_weak_strong} follows.
\end{proof}

\begin{remark}
    Note that we proved strong convergence of the sequence of operators $\mathcal{P}_n^*(\mathcal{P}_n\mathcal{L}\mathcal{P}_n^*)^k \mathcal{P}_n$. The proof, therefore, automatically upgrades the lemma to weak convergence (in the sense of measures) of projection-valued spectral measures in the strong operator topology.
\end{remark}

The above proof cannot be carried over to unbounded operators since the moments may not exist. Instead, we replace the polynomials appearing in the proof of \cref{lem:FS_moments} with rational functions. To do this, we must look at the resolvent.

\subsection{Strong convergence of the resolvent}

To deal with unbounded operators, we make the following assumption.

\begin{assumption}
    \label{core_assump}
    There exists a core $S\subset \mathcal{D}(\mathcal{L})$ such that
    $$
        \lim_{n\rightarrow\infty}\mathcal{L}\mathcal{P}_n^*\mathcal{P}_nu=\mathcal{L}u,\quad \lim_{n\rightarrow\infty}\mathcal{P}_n^*\mathcal{P}_nu=u\quad \forall u\in S.
    $$
\end{assumption}

Without such assumptions, spectral measures need not converge \cite[Theorem 2.3]{Levitin}. The following lemma shows that with this assumption, the resolvents of our projected operators converge strongly to the resolvent of $\mathcal{L}$. For a discussion of these kinds of results, see \cite[Chapter VIII]{kato2013perturbation} and \cite{weidmann1997strong,rellich1941storungstheorie}.

\begin{lemma}
    \label{strong_res_lemma}
    Suppose that \cref{core_assump} holds. Then
    $$
        \lim_{n\rightarrow\infty} \mathcal{P}_n^*[\mathcal{P}_n(\mathcal{L}-zI)\mathcal{P}_n^*]^{-1}\mathcal{P}_nv=(\mathcal{L}-zI)^{-1}v,\quad\forall v\in\mathcal{H},z\in\mathbb{C}\backslash\mathbb{R}.
    $$
\end{lemma}

\begin{proof}
    Fix $z\in\mathbb{C}\backslash\mathbb{R}$ and suppose first that there exists $u\in S$ with $(\mathcal{L}-zI)^{-1}v=u$. For notational convenience, let $T_n=[\mathcal{P}_n(\mathcal{L}-zI)\mathcal{P}_n^*]^{-1}$, which exists since $\spec(\mathcal{P}_n\mathcal{L}\mathcal{P}_n^*)\subset\mathbb{R}$. We may write
    $$
        \mathcal{P}_n^*\mathcal{P}_n(\mathcal{L}-zI)^{-1}v=\mathcal{P}_n^*\mathcal{P}_nu=\mathcal{P}_n^*T_n\mathcal{P}_n\mathcal{P}_n^*[\mathcal{P}_n(\mathcal{L}-zI)\mathcal{P}_n^*]\mathcal{P}_nu.
    $$
    This quantity converges to $(\mathcal{L}-zI)^{-1}v$, and hence it is enough to prove that
    $$
        \lim_{n\rightarrow\infty} \mathcal{P}_n^*T_n\mathcal{P}_nv-\mathcal{P}_n^*T_n\mathcal{P}_n\mathcal{P}_n^*[\mathcal{P}_n(\mathcal{L}-zI)\mathcal{P}_n^*]\mathcal{P}_nu=0.
    $$
    The quantity on the left is equal to
    $$
        \mathcal{P}_n^*T_n\mathcal{P}_n\left[(\mathcal{L}-zI)-\mathcal{P}_n^*\mathcal{P}_n(\mathcal{L}-zI)\mathcal{P}_n^*\mathcal{P}_n\right]u.
    $$
    Since $\|T_n\|\leq 1/|\mathrm{Im}(z)|$, \cref{core_assump} shows that this converges to zero. Hence, to prove the lemma, we must show that we can drop the assumption that $u\in S$. Let $u=(\mathcal{L}-zI)^{-1}v$ be general, then there exists a sequence $\{u_m\}\subset S$ with $\lim_{m\rightarrow\infty}u_m=u$ and $\lim_{m\rightarrow\infty}\mathcal{L}u_m=\mathcal{L}u$. Let $v_m=(\mathcal{L}-zI)u_m$, then we have just shown that
    $$
        \lim_{n\rightarrow\infty} \mathcal{P}_n^*[\mathcal{P}_n(\mathcal{L}-zI)\mathcal{P}_n^*]^{-1}\mathcal{P}_nv_m=(\mathcal{L}-zI)^{-1}v_m.
    $$
    Now $\lim_{m\rightarrow\infty}v_m=v$ and $\|T_n\|,\|(\mathcal{L}-zI)^{-1}\|\leq 1/|\mathrm{Im}(z)|$. Hence, the result also holds with $v$.
\end{proof}

\subsection{General convergence theorem}

We now have all of the tools needed to prove the convergence of spectral measures. We first prove the general result and then the result for Hermitian DMD. All the required assumptions are stated explicitly in the theorems.

\begin{theorem}
    \label{thm_gen_convergence}
    Let $\mathcal{L}$ be a self-adjoint operator on a Hilbert space $\mathcal{H}$ with domain $\mathcal{D}(\mathcal{L})$. Let $\{\mathcal{P}_n:\mathcal{H}\rightarrow \mathcal{V}_n\}_{n\in\mathbb{N}}$ be a sequence of orthogonal projections onto a Hilbert space $\mathcal{V}_n\subset \mathcal{H}$, such that $\mathcal{P}_n^*\mathcal{P}_n$ converges strongly to the identity and $\mathcal{V}_n\subset\mathcal{D}(\mathcal{L})$. In addition, suppose that \cref{core_assump} holds. Then, for any $v\in\mathcal{H}$ and any bounded continuous function $f$ on $\mathbb{R}$,
    \begin{equation}
        \label{wanted_conv}
        \lim_{n\rightarrow\infty}\int_{\mathbb{R}} f(\lambda) \d \mu_{v,n}(\lambda)=\int_{\mathbb{R}} f(\lambda) \d \mu_{v}(\lambda),
    \end{equation}
    where $\mu_{v,n}$ is the spectral measure of $\mathcal{P}_n\mathcal{L}\mathcal{P}_n^*$ w.r.t. the vector $v_n=\mathcal{P}_nv$.
\end{theorem}

\begin{proof}
    The idea of the proof is to use rational functions in the application of the Stone--Weierstrass theorem to first prove vague convergence of measures, and then upgrade the convergence to weak convergence using tightness. We first let $f(\lambda)=1/(\lambda-z)$ for $z\in\mathbb{C}\backslash\mathbb{R}$. For this choice, the functional calculus shows that
    \begin{align*}
        \int_\mathbb{R} f(\lambda)\dd\mu_{v,n}(\lambda) & =\langle\mathcal{P}_n^*[\mathcal{P}_n(\mathcal{L}-zI)\mathcal{P}_n^*]^{-1} \mathcal{P}_n v,v\rangle, \\
        \int_\mathbb{R} f(\lambda)\dd\mu_v(\lambda)     & =\langle (\mathcal{L}-zI)^{-1} v,v\rangle.
    \end{align*}
    \cref{strong_res_lemma} now shows that \eqref{wanted_conv} holds for this particular choice of $f$. The Stone--Weierstrass theorem (for locally compact Hausdorff spaces) then implies that \eqref{wanted_conv} holds if $\lim_{|x|\rightarrow\infty}f(x)=0$. To finish the proof, note that $\lim_{n\rightarrow\infty}\mu_{v,n}(\mathbb{R})=\mu_v(\mathbb{R})=\|v\|^2$, which implies weak convergence given the vague convergence.
\end{proof}

The convergence of Hermitian DMD now follows almost immediately. For a given observable $g\in L^2(\Omega,\omega)$, we define
$$
    \gv_{N,M}= (\Wv^{1/2}\Psiv_X)^{\dagger}\Wv^{1/2}\left(g(\xv^{(1)}),\ldots,g(\xv^{(M)})\right)^\top.
$$
such that
$$
    \lim_{M\rightarrow\infty}\gv_{N,M}=\gv_{N},\quad\text{with}\quad \mathcal{P}_Ng=\Psiv \gv_{N},
$$
\rev{where $\Psiv$ is the dictionary feature map in \cref{feat_map_def}.}
We emphasize that this is the standard way to compute expansion coefficients in the EDMD algorithm.

\begin{theorem}[Convergence of Hermitian DMD]
    \label{thm:convergence_of_hermitianDMD}
    Consider the space of functions $V_N=\mathrm{span}\{\psi_1,\ldots,\psi_N\}$ and let $\mathcal{P}_N:L^2(\Omega,\omega)\rightarrow V_N$ be the corresponding orthogonal projection. Suppose that $\mathcal{P}_N^*\mathcal{P}_N$ converges strongly to the identity, the Koopman operator $\mathcal{K}$ is self-adjoint, the quadrature rule converges as in \eqref{quad_convergence}, and that
    $$
        \mathcal{K}u= \lim_{N\rightarrow\infty}\mathcal{P}_N\mathcal{K}\mathcal{P}_N^*\mathcal{P}_N u,\quad  u\in\mathcal{D}(\mathcal{K}).
    $$
    Then, for any observable $g\in L^2(\Omega,\omega)$ and any bounded continuous function $f$ on $\mathbb{R}$,
    \begin{equation}
        \label{eq_conv_of_HDMD}
        \lim_{N\rightarrow\infty}\lim_{M\rightarrow\infty} \int_{\mathbb{R}} f(\lambda)\d \mu_{\gv_{N,M}}^{(M)}(\lambda)= \int_{\mathbb{R}} f(\lambda)\d \mu_g(\lambda),
    \end{equation}
    were $\mu_{\gv_{N,M}}^{(M)}$ is the spectral measure of $\Psiv \gv_{N,M}$ with respect to the Hermitian DMD matrix.
\end{theorem}

\begin{remark}
    The use of Hermitian DMD is crucial in \cref{thm:convergence_of_hermitianDMD} to restrict the least squares problem in \eqref{symmetric_easy}, and ensure that the spectral measures $\mu_{\gv_N}^{(M)}$ are supported on $\mathbb{R}$.
\end{remark}

\begin{proof}[Proof of \cref{thm:convergence_of_hermitianDMD}]
    The convergence of Hermitian DMD as $M\to \infty$ implies that
    $$
        \lim_{M\rightarrow\infty} \int_{\mathbb{R}} f(\lambda)\d \mu_{\gv_{N,M}}^{(M)}(\lambda)= \int_{\mathbb{R}} f(\lambda)\d \mu_{\gv_N}(\lambda),
    $$
    where $\mu_{\gv_N}$ is the spectral measure of $\mathcal{P}_N\mathcal{K}\mathcal{P}_N^*$ with respect to $\mathcal{P}_Ng$. The rest of the proof follows directly from \cref{thm_gen_convergence}.
\end{proof}

\section{Numerical example}
\label{sec:computational_example}

In this section, we evaluate the convergence of the Hermitian DMD algorithm on the two-dimensional Schr\"odinger equation given by
\[
    i\frac{\partial u}{\partial t} =\hat{H}(u)=-\frac{1}{2}\Delta u+V(x,y)u, \quad (x,y)\in \R^2,
\]
where $V(x,y)=(x^2+y^2)/2$ is an external harmonic potential, $u:\Omega\to \mathbb{C}$ is a normalized wave function, and $\hat{H}$ is the Hamiltonian operator of the system. Here, we select a computational domain $\Omega=(-5,5)^2$ to be large enough so that solutions vanish well before reaching the boundary. The snapshot data contain functions of the form $(u,i\partial_t u)$, which are related by the self-adjoint Hamiltonian operator $\hat{H}$. We use $N=400$ initial conditions consisting of Gaussian bumps (the dictionary) inside $\Omega$ as
\[u(x,y) = (1+i)e^{-3((x-x_i)^2+(y-y_j)^2)},\quad 1\leq i,j\leq 20,\]
where the Gaussian centers $(x_i,y_j)$ are uniformly distributed in the domain $[-4,4]^2$.\footnote{\rev{For suitable Gaussian bump functions, one can easily show that \cref{core_assump} is satisfied for our Schr\"odinger operator.}} We then evaluate the functions on the domain $\Omega$ at a uniform tensor-product grid of $M=300^2$ snapshot points, and discretize the integral in \eqref{eq:ContinuousLeastSquaresProblem} using a trapezoidal rule. Finally, we compute the $N\times N$ Koopman matrix $\Kv$ using the Hermitian DMD algorithm described in \cref{sec_self_adjoint}, and perform an eigenvalue decomposition of $\Kv$ to obtain the first eigenstates and corresponding eigenvalues to the Schr\"odinger operator (see~\cref{fig_schrodinger}).

\begin{figure}[htbp]
    \centering
    \begin{overpic}[width=\textwidth]{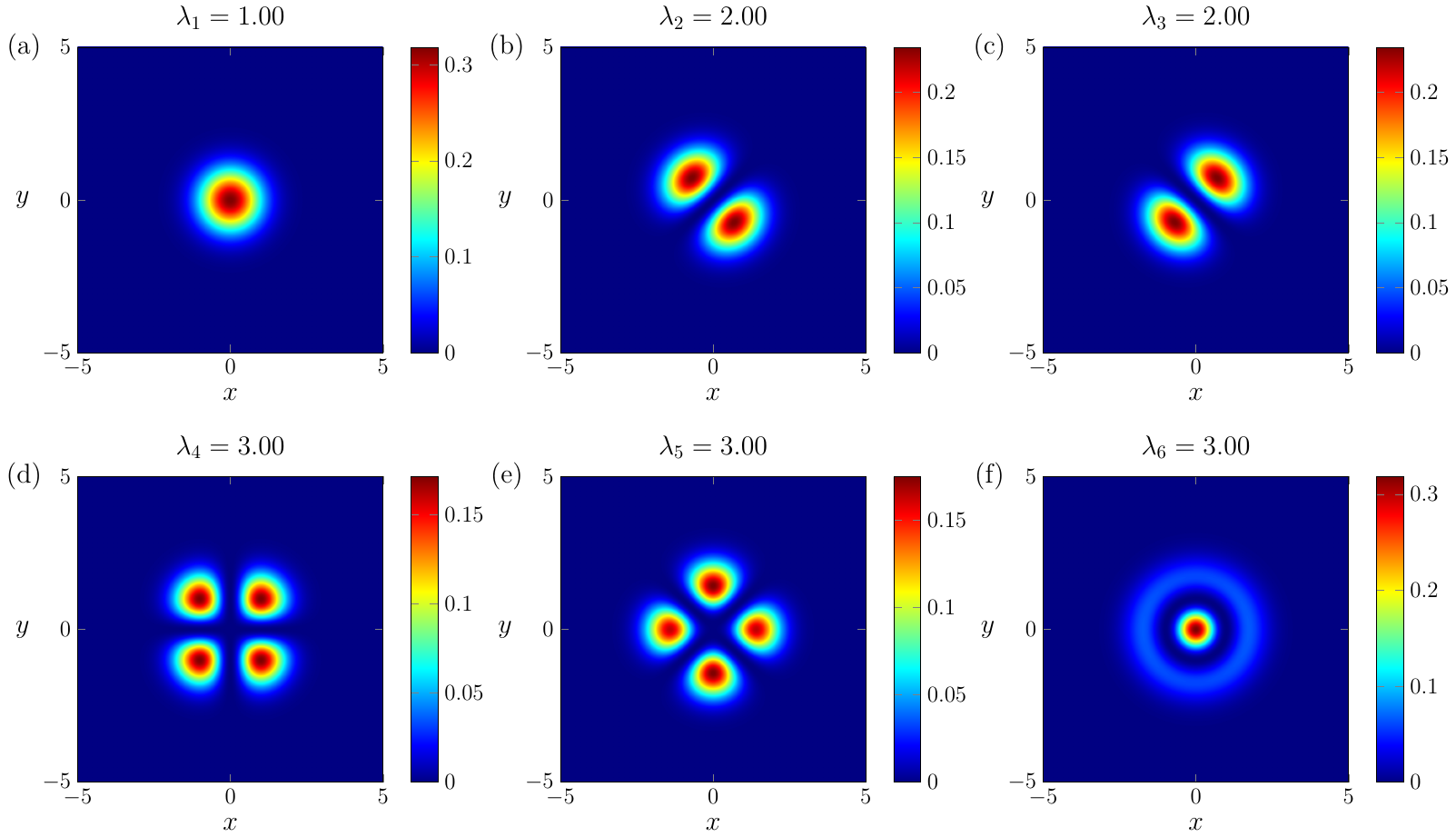}
    \end{overpic}
    \caption{The first $6$ energy eigenstates and eigenvalues of the Schr\"odinger operator discovered by Hermitian DMD.}
    \label{fig_schrodinger}
\end{figure}

The analytical expression for the eigenmodes of the Schr\"odinger equation are obtained using the standard ansatz $u(x,y,t)=\phi(x,y)e^{-iEt}$, where $\phi$ is the eigenfunction and $E$ is the corresponding eigenvalue. Following a separation of variables, one finds that the eigenmodes are proportional to~\cite[Eq.~(8)]{charalampidis2018computing}
\[\phi_{m,n}(x,y) = H_m(x)H_n(y)e^{-(x^2+y^2)/2},\quad (x,y)\in \R^2,\quad m,n\geq 0,\]
where $H_m$ is the Hermite polynomial of degree $m$, and the corresponding eigenvalues are given by $E_{m,n}=m+n+1$. We report in \cref{fig_eigenvalue_schrodinger} the first hundred eigenvalues obtained by Hermitian DMD along with the exact ones and observe a good agreement up to the $50$th eigenvalue.

\begin{figure}[htbp]
    \centering
    \begin{overpic}[width=0.6\textwidth]{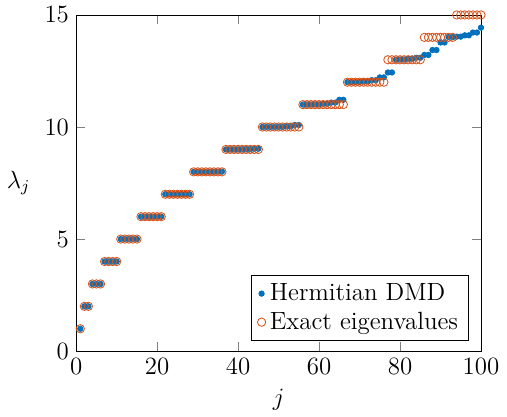}
    \end{overpic}
    \caption{The first $100$ eigenvalues of the Schr\"odinger operator discovered by Hermitian DMD along with the exact ones.}
    \label{fig_eigenvalue_schrodinger}
\end{figure}

\begin{figure}[htbp]
    \centering
    \begin{overpic}[width=\textwidth]{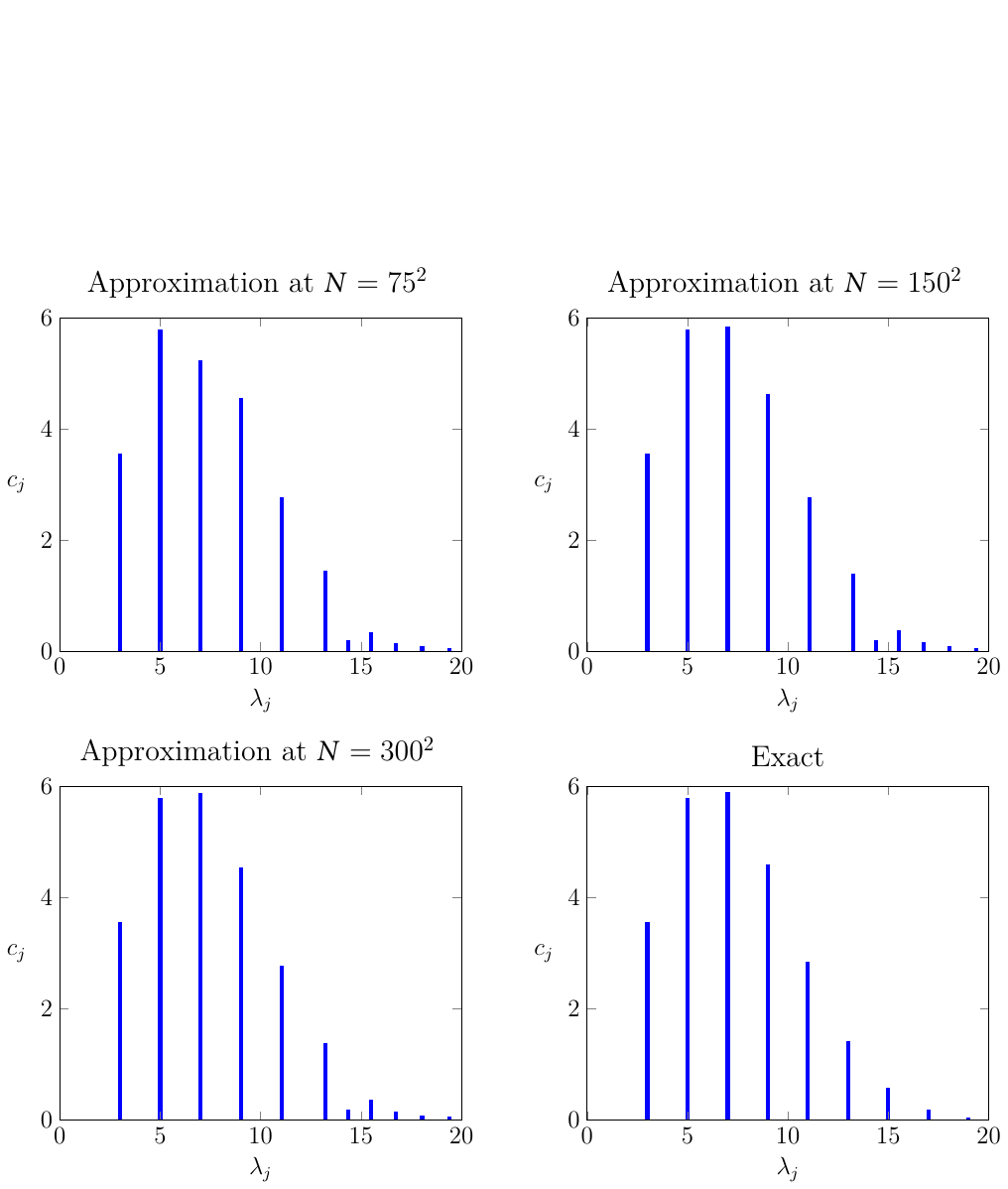}
    \end{overpic}
    \caption{Visualization of the approximate measures in \eqref{approx_meas}. For this example, weak convergence in \cref{thm:convergence_of_hermitianDMD} means that the positions and heights of the spikes converge. The heights of the spikes, $c_j$, can be thought of as an energy distribution akin to a Fourier transform (but now provided by the spectral theorem).}
    \label{fig_schrodinger_spectral}
\end{figure}

Then, we estimate $\mu_f$ for the function $f$ defined as
\[f(x,y) = \sin(\pi x/5)\sin(\pi y/5),\quad (x,y)\in [-5,5]^2.\]
The measure $\mu_{f,n}$ is given as a sum of Dirac measures
\begin{equation}
    \label{approx_meas}
    \mu_{f,N}=\sum_{j=1}^N c_j\delta_{\lambda_j},\quad c_j=|\vv_j^*\Gv\fv|^2.
\end{equation}
Here, $(\lambda_j,\vv_j)$ are the eigenpairs computed by Hermitian DMD and $f=\Psiv\fv$. For comparison, for each analytic eigenvalue, we take the mean of the cluster of the eigenvalues $\lambda_j$ that approximate it and then sum the weights $c_j$. The results are shown in \cref{fig_schrodinger_spectral,tab:my_label} and demonstrate the convergence in \cref{thm:convergence_of_hermitianDMD}.

\begin{table}[t!]
    \centering
    \caption{Convergence of $c_j$ ($\lambda_j$) in the first five spikes of \cref{fig_schrodinger_spectral} for different values of $N$.}
    \begin{tabular}{l|c|c|c|c|c}
        \hline
        $N = 75^2$  & 3.56 (3.00) & 5.79 (5.00) & 5.23 (7.00) & 4.55 (9.01) & 2.77 (11.06) \\
        $N = 150^2$ & 3.56 (3.00) & 5.79 (5.00) & 5.85 (7.00) & 4.62 (9.01) & 2.77 (11.06) \\
        $N = 300^2$ & 3.56 (3.00) & 5.79 (5.00) & 5.90 (7.00) & 4.59 (9.01) & 2.86 (11.06) \\
        \hline\hline
        Exact       & 3.56 (3.00) & 5.79 (5.00) & 5.90 (7.00) & 4.59 (9.00) & 2.86 (11.00) \\
        \hline
    \end{tabular}

    \label{tab:my_label}
\end{table}

\rev{
    \section*{Data availability}
    The code to reproduce the numerical experiments is publicly available on GitHub at \url{https://github.com/NBoulle/ConvergenceHermitianDMD}.
}

\section*{Acknowledgments}
NB was supported by the SciAI Center funded by the Office of Naval Research (ONR), under Grant Number N00014-23-1-2729.

\bibliographystyle{elsarticle-num}
\bibliography{biblio}

\end{document}